\documentclass[11pt, oneside]{amsart}
 \usepackage[text={6.5in,9in},centering,letterpaper,dvips]{geometry}
\usepackage{graphicx}
\usepackage{amsfonts}
\usepackage[dvipsnames]{xcolor}
\usepackage{subcaption}
\usepackage{epsf}
\usepackage{amssymb}
\usepackage{amsmath}
\usepackage{amscd}
\usepackage{pdfpages}
\usepackage{fancyhdr}
\usepackage{setspace}
\usepackage{soul} 
\usepackage[all]{xy}
\usepackage{verbatim}
\usepackage{enumerate}
\usepackage[colorlinks=true, urlcolor=NavyBlue, linkcolor=NavyBlue, citecolor=NavyBlue]{hyperref}

\theoremstyle{theorem}
\newtheorem{theorem}{Theorem}
\newtheorem{proposition}[theorem]{Proposition}
\newtheorem{lemma}[theorem]{Lemma}
\newtheorem{question}[theorem]{Question}

\newtheorem{conjecture}[theorem]{Conjecture}

\makeatletter
\newtheorem*{rep@theorem}{\rep@title}
\newcommand{\newreptheorem}[2]{%
\newenvironment{rep#1}[1]{%
 \def\rep@title{#2 \ref{##1}}%
 \begin{rep@theorem}}%
 {\end{rep@theorem}}}
\makeatother

\newreptheorem{theorem}{Theorem}
\newreptheorem{lemma}{Lemma}
\newreptheorem{question}{Question}
\newreptheorem{corollary}{Corollary}
\newreptheorem{proposition}{Proposition}
\newreptheorem{conjecture}{Conjecture}

\theoremstyle{definition}

\newcommand{\Aa}{\mathcal A}
\newcommand{\Bb}{\mathcal B}

\makeatletter
\def\@seccntformat#1{%
  \protect\textup{\protect\@secnumfont
    \ifnum\pdfstrcmp{subsection}{#1}=0 \bfseries\fi
    \csname the#1\endcsname
    \protect\@secnumpunct
  }%
}  
\makeatother

\begin{document}

\rhead{\thepage}
\lhead{\author}
\thispagestyle{empty}


\raggedbottom
\pagenumbering{arabic}
\setcounter{section}{0}


\title{Positive braids minimize ascending number}

\author{Lowell Davis}
\address{Department of Mathematics, Western Washington University, Bellingham, WA}
\email{fdavis.lowell@gmail.com}
\urladdr{} 

\author{Jeffrey Meier}
\address{Department of Mathematics, Western Washington University, Bellingham, WA}
\email{jeffrey.meier@wwu.edu}
\urladdr{http://jeffreymeier.org} 

\begin{abstract}
	A well-known algorithm for unknotting knots involves traversing a knot diagram and changing each crossing that is first encountered from below.
	The minimal number of crossings changed in this way across all diagrams for a knot is called the \emph{ascending number} of the knot.
	The ascending number is bounded below by the unknotting number.
	We show that for knots obtained as the closure of a positive braid, the ascending number equals the unknotting number.
	
	We also present data indicating that a similar result may hold for positive knots.
	We use this data to examine which low-crossing knots have the property that their ascending number is realized in a minimal crossing diagram, showing that there are at most 5 hyperbolic, alternating knots with at most 12 crossings with this property.
\end{abstract}

\maketitle

\section{Introduction}
\label{sec:into}

The unknotting number of a knot is one of the oldest knot invariants.
It was introduced as early as 1877 by Peter Guthrie Tait, who referred to it as the ``beknottedness'' of a knot.
Nearly 150 years later, many basic aspects of the invariant remain unknown.
For example, it is unknown whether the unknotting number is additive under connected sum~\cite{Scharlemann:unknotting}, and there are simple knots, such as $\mathbf{10_{11}}$, for which the value of the invariant is unknown~\cite{KnotInfo}.
For another testament to the elusiveness of this invariant, recall that, in 1992, Kronheimer and Mrowka~\cite{Kronheimer-Mrowka:gauge} gave a lower-bound for the unknotting number of torus knots using gauge theory, settling a 25-year-old conjecture of Milnor.
(Upper-bounds were previously obtained using more geometric methods~\cite{Rudolph:Seifert,Boileau-Weber:noeds}.)

There is a well-known algorithm to change any knot diagram into a diagram for the unknot, which we refer to as the \emph{roller-coaster algorithm}, that consists of traversing a knot diagram and changing every crossing that is first encountered along a strand passing below the crossing.
The resulting diagram is called an \emph{ascending diagram}.
The \emph{ascending number}  $a(K)$ of a knot $K$ is the fewest number of crossings changed when the roller-coaster algorithm is applied to a diagram of $K$.
Clearly, we have $a(K)\geq u(K)$ for any knot $K$.
This invariant was introduced by Ozawa~\cite{Ozawa:ascending}, who established a number of foundational results\footnote{This invariant was also studied by Shimizu, who called it the \emph{warping degree}~\cite{Shimizu:warping}.}.
In particular, he showed that $a(T_{p,q}) = u(T_{p,q})$ for any torus knot $T_{p,q}$.

In this paper, we generalize this result of Ozawa to the class of knots obtained as the closure of positive braids.

\begin{theorem}
\label{thm:positive_braid}
	If $K$ is the closure of a positive braid, then $a(K) = u(K)$.
\end{theorem}

Our methods are constructive and give upper-bounds for the ascending number; the relevant lower-bounds come from work of Rudolph that gives a reinterpretation of the result of Kronheimer and Mrowka to the setting of quasipositive knots~\cite{Rudolph:QP}.
Theorem~\ref{thm:positive_braid} is similar to Theorem~1.2 of~\cite{Kegel-Lewark-et-Al:unknotting}.
If $a(K) = u(K)$, we say that $K$ has the \emph{roller-coaster property}.
Ozawa showed that $a(K)=1$ if and only if $K$ is a twist knot, so most 2--bridge knots (for example) do not have the roller-coaster property.

We also consider the question of whether the unknotting number of a given knot is obtained by applying the roller-coaster algorithm to a minimal crossing diagram of the knot.
We say that such knots have the \emph{strong roller-coaster property}.
It is an easy exercise to check that, among twist knots, only the trefoil and figure-eight knot have the strong roller-coaster property.
We establish the following sufficient condition for a positive braid closure to have the strong roller-coaster property.

\begin{theorem}
\label{thm:positive_strong}
	Suppose that $K$ is the closure of a positive $n$--braid and the braid index of $K$ is $n$.
	Then, $K$ has the strong roller-coaster property.
\end{theorem}

Stoimenow showed that there exist knots with braid index four that are the closure of positive braids but not the closure of positive, 4--stranded braids~\cite[Example~7]{Stoimenow:crossing_number}, so the proof given above cannot show that every positive braid closure has the strong roller-coaster property.
However, the two examples given by Stoimenow, which have unknotting number six, \emph{do} have the strong roller-coaster property, it turns out; see Figure~\ref{fig:stoimenow_rc}, below.
This raises the following question.

\begin{question}
\label{ques:strong_positive}
	Does every positive braid closure have the strong roller-coaster property?
\end{question}

Finally, we use a SnapPy~\cite{SnapPy} computer program to investigate which knots with low crossing number have the roller-coaster property or strong roller-coaster property.
Our results are summarized in Tables~\ref{tab:8} and~\ref{tab:9}; see Theorems~\ref{thm:9cross} and~\ref{thm:alt_12} below.

Every positive knot in these tables has the roller-coaster property.
This motivates the following question regarding whether Theorem~\ref{thm:positive_braid} generalizes to positive knots.

\begin{question}
\label{ques:sqp}
	Does every positive knot have the roller-coaster property?
\end{question}

In contrast, $\mathbf{8_{20}}$ and $\mathbf{8_{21}}$ are quasipositive, but do not have the roller-coaster property, and many Whitehead doubles of strongly quasipositive knots are strongly quasipositive~\cite{Rudolph:QP}, but they have unknotting number one, hence cannot have the roller-coaster property~\cite{Fung:immersions}.

Using our data, we were also able to improve the known upper bound on the ascending number for three 9--crossing knots, determining the ascending number for two of these.
By restricting to alternating knots, we obtain the following.

\begin{theorem}
\label{thm:alt_12}
	Among the 1,846 alternating, hyperbolic knots with at most 12 crossings, at most five have the strong roller-coaster property:
	\begin{enumerate}
		\item The knots $\mathbf{4_1}$, $\mathbf{8_{12}}$, and $\mathbf{8_{18}}$ have the strong roller-coaster property; and
		\item The knots $\mathbf{12a_{181}}$ and $\mathbf{12a_{477}}$ have the strong roller-coaster property if their unknotting number turns out to be 3, rather than 2.
	\end{enumerate}
\end{theorem}

The above theorem verifies the following conjecture up to 12 crossings.

\begin{conjecture}
\label{conj:alt}
	If $K$ is a hyperbolic, alternating knot with the strong roller-coaster property, then the crossing number of $K$ is divisible by four.
\end{conjecture}

Related to this is the following conjecture.
Let $a^{min}(c)$ denote the smallest ascending number obtained by a knot with a reduced alternating diagram with $c$ crossings.

\begin{conjecture}
\label{conj:min-degree}
	$a^{min}(c) =\lceil c/4\rceil$.
\end{conjecture}

This conjecture was confirmed up to 12 crossings by Ohya and Shimizu~\cite[Table~1]{Ohya-Shimizu:lower} and is verified up to 16 crossings by our computer program.

\subsection*{Organization}

Section~\ref{sec:proofs} contains the proofs of Theorems~\ref{thm:positive_braid} and~\ref{thm:positive_strong} concerning positive braid closures.
Section~\ref{sec:low} contains the proof of Theorem~\ref{thm:alt_12} and a discussion of the data on low-crossing knots produced by our computer program, which is available at \url{http://jeffreymeier.org}.

\subsection*{Acknowledgements}

We are indebted to Evan Scott for alerting us of the literature on ascending number and warping degree initiated by Ozawa and Shimizu, respectively, and we are grateful to Sebastian Baader for an email correspondence regarding this literature.
We thank Matt Hedden for pointing out that Whitehead doubles can be strongly quasipositive, correcting an oversight in the first version of this paper.
The second author was supported by NSF grant DMS-2405324.

\section{Proofs of Theorems~\ref{thm:positive_braid} and~\ref{thm:positive_strong}}
\label{sec:proofs}

Let $\beta$ be a positive braid on $n$--strands, and let $D$ be the standard diagram for the braid, drawn from left to right with a crossing for each Artin generator.
Identifying the left and right sides of the diagram $D$ gives a diagram $\widehat D$ for the braid closure $\widehat\beta$.
An example is shown in Figure~\ref{fig:pos_braid_rc}.
Choose as a starting point for the roller-coaster algorithm the point $p$ at the top-left of the diagram $D$, and orient the knot to the right near $p$.

\begin{figure}[ht!]
	\centering
	\includegraphics[width=.9\textwidth]{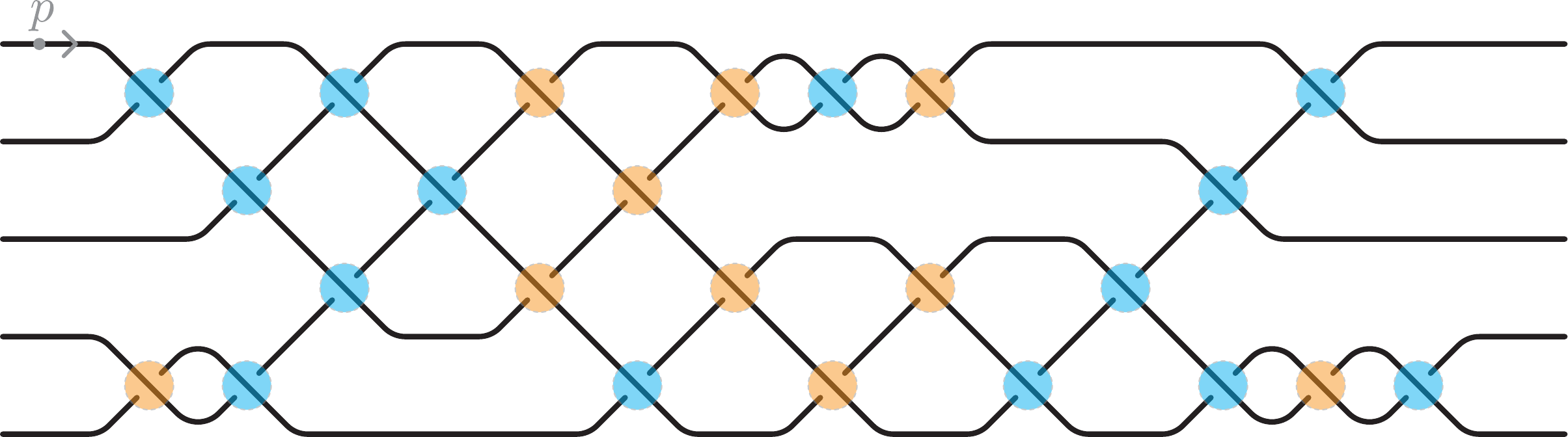}
	\caption{The roller-coaster algorithm realizes the unknotting number of closures of positive braids.}
	\label{fig:pos_braid_rc}
\end{figure}

Running the roller-coaster algorithm on $\widehat D$ starting at $p$ and according to the indicated orientation, let $\Aa$ denote the set of crossings that are first encountered from above, and let $\Bb$ denote the set of crossings that are first encountered from below.

\begin{lemma}
\label{lem:pos_braid}
	$|\Aa| - |\Bb| = n-1$.
\end{lemma}

\begin{proof}
	We induct on the number of crossings $C$ in the diagram $D$.
	For the base case, since $\widehat D$ is a diagram for a knot, there must be at least $n-1$ crossings, and each Artin generator must occur at least once.
	Thus, the base case is the positive braid diagram on $n$--strands for the torus knot $T(n,1)$.
	It is clear in this case that $|\Aa|= n-1$ and $|\Bb| = 0$, as desired.
	
	Assume for the inductive hypothesis that the lemma holds for any diagram with fewer than $C$ crossings, and assume that $D$ has $C$ crossings.
	Two strands $\gamma$ and $\gamma'$ that intersect in two crossings $c$ and $c'$ form a bigon: the union of the portions of $\gamma$ and $\gamma'$ between these crossings.
	A bigon is innermost if it contains no other bigons.
	Our inductive step takes two different forms, depending on whether or not a bigon is present.
	
	First, assume a bigon is present.
	This is equivalent to the presence of two strands of $D$ that cross more than once.
	Let $B$ be an innermost bigon, and denote the vertices of $D$ by $c$ and $c'$.
	Let $\gamma$ and $\gamma'$ be the strands of $D$ crossing at $c$ and $c'$, hence forming the innermost bigon.
	Let $\omega$ and $\omega'$ denote the portions of $\gamma$ and $\gamma'$, respectively, lying between $c$ and $c'$.
	See Figure~\ref{fig:bigon_rc} for an illustrative example of the following discussion.
	
\begin{figure}[ht!]
	\centering
	\includegraphics[width=.9\textwidth]{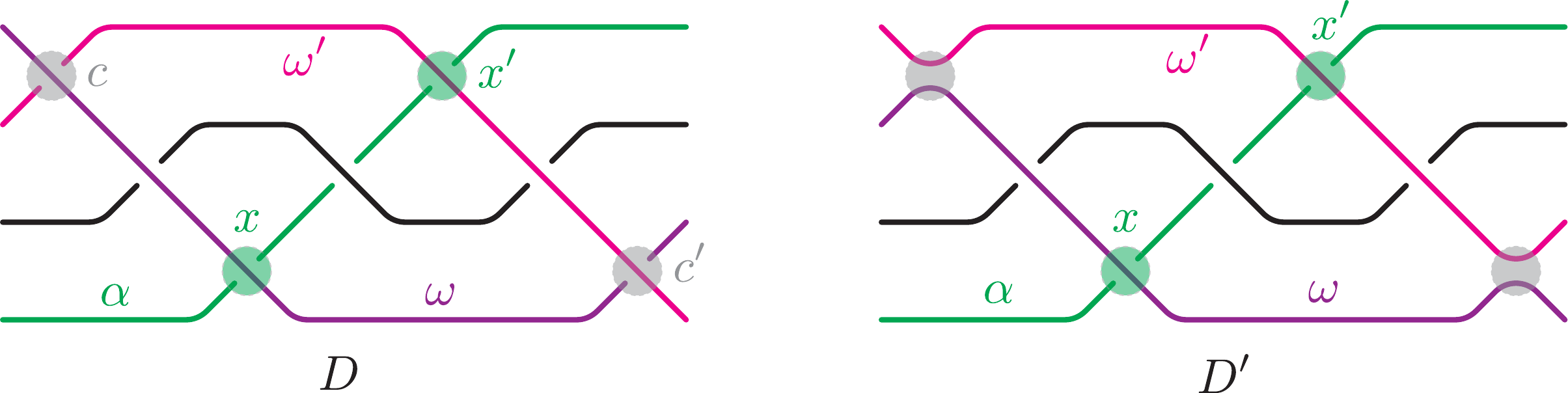}
	\caption{Working with a bigon in a positive braid diagram}
	\label{fig:bigon_rc}
\end{figure}
	
	Since $B$ is innermost, any other strand of $D$ that intersects $B$ must traverse it, entering by crossing one edge of $B$ and exiting via the other edge.
	Since $D$ is a positive braid diagram, if a strand $\alpha$ of $D$ enters the bigon $B$ by passing under (respectively, over) one edge of $B$, then it exits by passing under (respectively, over) the other edge of $B$.
	It follows $\alpha$ only passes over or only passes under at every crossing in $B$.
	
	Now, let $D'$ be the diagram obtained by smoothing the crossings $c$ and $c'$ according to the orientation of $D$.
	The roller-coaster algorithm on $D$ induces one on $D'$ that has an identical effect on every crossing of $D$ other than those along $\alpha$.
	For each crossing $x$ on $\alpha$, there is a corresponding crossing $x'$ on $\alpha$.
	Because the algorithm on $D'$ traverses $\omega$ and $\omega'$ in the opposite order from the algorithm on $D$, it follows that the algorithm on $D'$ changes $x$ (respectively, $x'$) if and only if the algorithm on $D$ changes $x'$ (respectively, $x$).
	
	Let $\Aa'$ (respectively, $\Bb'$) denote the crossings that are first encountered from above (respectively, below) during the induced roller-coaster algorithm on $D'$.
	Note that one of $c$ and $c'$ lies in each of $\Aa'$ and $\Bb'$.
	It follows that $|\Aa'| = |\Aa|-1$ and $|\Bb'| = |\Bb|-1$.
	By the inductive hypothesis, since $D'$ has $C-2$ crossings, the roller-coaster algorithm applied to $D'$ has the property that $|\Aa'|-|\Bb'| = n-1$.
	The claim of the lemma for $D$ follows.
	
	Now suppose that $D$ contains no bigon, so each strand of $D$ intersects every other strand in at most one point.
	If this is the case, then each strand of $D$ can be arranged to be increasing or decreasing.
	Denote the strands by $\gamma_1,\ldots,\gamma_n$ based on the order they are traversed, and let $\gamma_i$ be the first strand that is increasing, rather than decreasing.
	(Of course, $\gamma_1$ must be decreasing.)
	Let $c$ be the crossing involving $\gamma_i$ and $\gamma_{i-1}$, which must exist since $\gamma_{i-1}$ is decreasing and ends at the same height that $\gamma_i$, which is increasing, starts.
	Note that $c$ was first encountered from above along $\gamma_{i-1}$; see Figure~\ref{fig:resolution_rc}, where $i=3$.
	
	Let $D'$ be the diagram obtained by resolving $c$ according to the orientation.
	Let $\alpha$ denote the strand of $D'$ containing the initial point of $\gamma_i$.
	Note that the endpoints of $\alpha$ lie at the same height; the closure of $D'$ has two components.
	Every crossing encountered when traversing $\alpha$ from its left endpoint is encountered in the same way (from above or from below) as when we traversed $\gamma_{i-1}$ and $\gamma_i$, though the order of these encounters may be cyclically permuted.
	Since $D'$ is still a positive diagram, half of the crossings encountered along $\alpha$ are from above, while the other half are from below; suppose there are $m$ of each.
	
\begin{figure}[ht!]
	\centering
	\includegraphics[width=.75\textwidth]{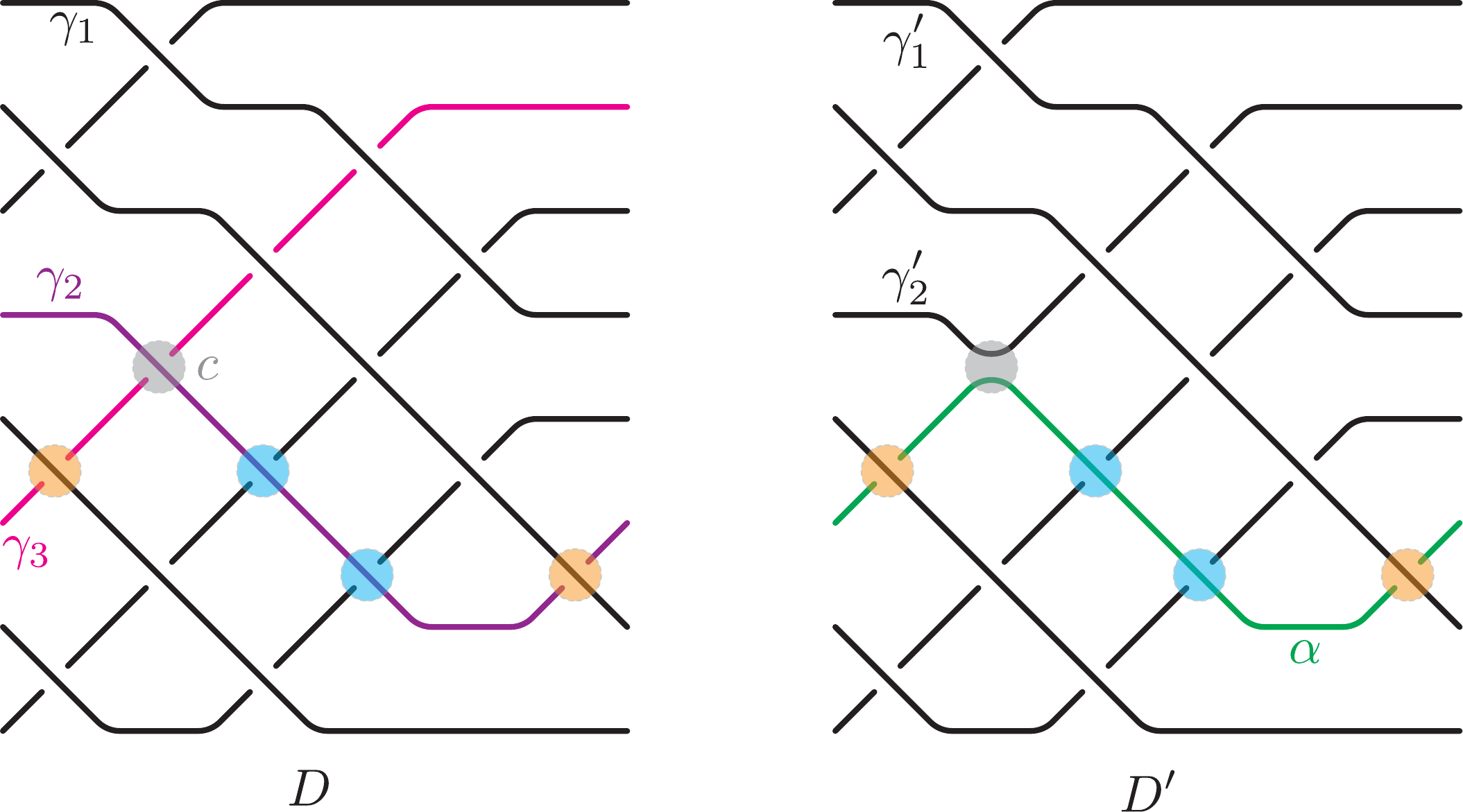}
	\caption{Resolving a crossing in a positive braid diagram with no bigons}
	\label{fig:resolution_rc}
\end{figure}
	
	Let $D''$ denote the diagram obtained by removing the strand $\alpha$ completely.
	Note that $D''$ is a $(n-1)$--strand braid; let $\Aa''$ (respectively, $\Bb''$) denote the crossings of $D''$ first encountered from above (respectively, below) during the induced roller-coaster algorithm on $D''$.
	We have $|\Aa''| = |\Aa|- m -1$ and $|\Bb''| = |\Bb| - m$.
	Since $D''$ has fewer crossings than $D$, our inductive hypothesis applies, so $|\Aa''| - |\Bb''| = n-2$.
	The claim of the lemma for $D$ follows.
\end{proof}

\begin{proof}[Proof of Theorem~\ref{thm:positive_braid}]
	Let $K$ be the closure of a positive braid $\beta$ with $n$ strands, and let $D$ be a diagram for $\beta$ with $C$ crossings.
	By Lemma~\ref{lem:pos_braid}, running the roller-coaster algorithm from left to right, starting at the top-left involves first encountering $|\Aa|$ crossings from above and $|\Bb|$ crossings from below, with $|\Aa| - |\Bb| = n-1$.
	So, $a(K)\leq |\Bb|$.
	
	On the other hand, Rudolph proved that
	$$u(K) = g(K) = \frac{C - n + 1}{2},$$
	giving the lower-bound in~\cite{Rudolph:QP} and the upper-bound in~\cite{Rudolph:Seifert}.
	Since $C = |\Aa| + |\Bb|$, this means
	$$u(K) = \frac{|\Aa| + |\Bb| - (|\Aa| - |\Bb|)}{2} = |\Bb|.$$
	So, $a(K) = u(K)$.
\end{proof}

\begin{proof}[Proof of Theorem~\ref{thm:positive_strong}]
	The proofs of Lemma~\ref{lem:pos_braid} and Theorem~\ref{thm:positive_braid} indicate that the roller-coaster algorithm can be applied to any positive braid diagram $D$ of $K$ to realize the unknotting number: one simply needs to start at the top-left corner of the braid diagram, as indicated in Figure~\ref{fig:pos_braid_rc}.
	By hypothesis, we can assume that $D$ minimizes the braid index of $K$.
	Murasugi proved that positive braid diagrams minimize the crossing number of $K$ when they minimize the braid index of $K$~\cite[Proposition~7.4]{Murasugi:braid_index}.
	So, we have that $D$ minimizes the crossing number of $K$, as desired.
\end{proof}

We conclude this section by giving the roller-coaster unknottings of Stoimenow's knots mentioned in the introduction; see Figure~\ref{fig:stoimenow_rc}.

\begin{figure}[ht!]
	\centering
	\includegraphics[width=.75\textwidth]{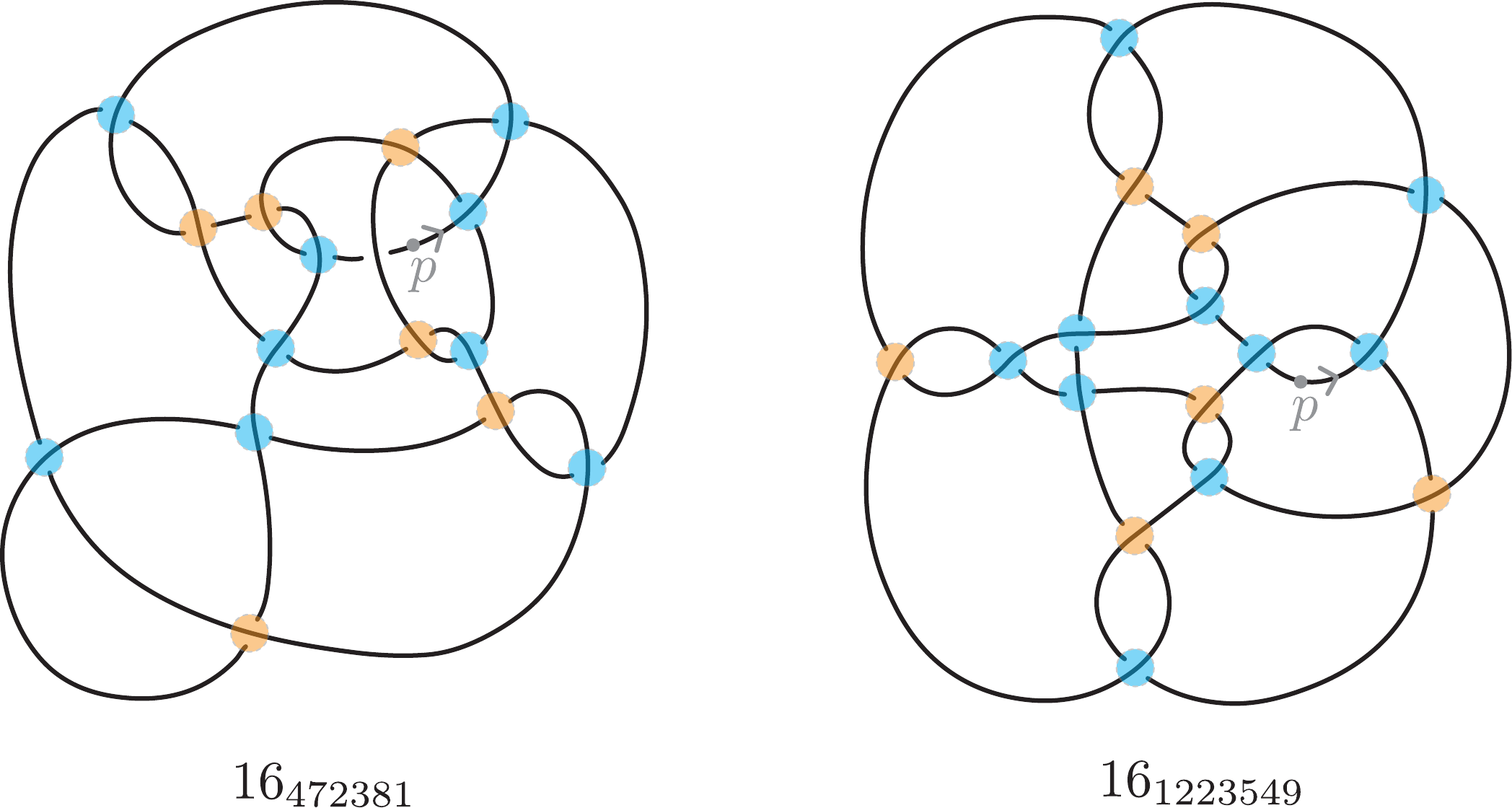}
	\caption{Stoimenow's knots have the strong roller-coaster property despite not being the closures of positive braids realizing their braid index.}
	\label{fig:stoimenow_rc}
\end{figure}

\section{Low-crossing knots}
\label{sec:low}

We created a computer program that ran the roller-coaster algorithm in every possible way for every knot diagram with at most 11 crossings.
Dowker-Thistlethwaite codes for the diagrams realizing the lowest ascending number found are given in the penultimate columns of Tables~\ref{tab:8} and~\ref{tab:9}.
Most noteworthy, our computer program found 11--crossing diagrams improving known upper bounds for three 9--crossing knots, which determines the ascending number for two of them.

\begin{proposition}
\label{prop:upper}
	The knots $\mathbf{9_{32}}$ and $\mathbf{9_{33}}$ have ascending number 2.
	The knot $\mathbf{9_{40}}$ has ascending number at most 3.
\end{proposition}

We are able to determine which knots with at most 9 crossings have the roller-coaster property or the strong roller-coaster property, with a few exceptions.

\begin{theorem}
\label{thm:9cross}
	Among the 84 prime knots with at most 9 crossings:
	\begin{enumerate}
		\item exactly 12 have the strong roller-coaster property;
		\item at least 32 others have the roller-coaster property;
		\item at least 34 others have neither; and
		\item there are 6 knots that may or may not have the roller-coaster property, but it they do, a diagram with at least 12 crossings is required.
	\end{enumerate}
\end{theorem}

\begin{proof}
	The claimed result are indicated in Tables~\ref{tab:8} and~\ref{tab:9}, which agree with the tables from~\cite{Jablan:families}, with the following exceptions.
	The DT codes given in the seventh column offer improved upper bounds for $\mathbf{9_{32}}$, $\mathbf{9_{33}}$, and $\mathbf{9_{40}}$.
	This proves that the first two of these have ascending number 2 and improves the upper bound for the ascending number of the third by one.
	Also, for $\mathbf{8_2}$, $\mathbf{8_5}$, $\mathbf{9_{11}}$, $\mathbf{9_{20}}$, $\mathbf{9_{36}}$, and $\mathbf{9_{43}}$, our table includes the improved lower bounds that were given by Higa using the Conway polynomial~\cite{Higa:conway}.
	Otherwise, if $u(K)\geq 2$, then the relevant lower bound either comes from the unknotting number (indicated by $u(K)$), and if $u(K)=1$ and $K$ is not a twist knot, then the relevant lower bound comes from the fact that the only knots with ascending number one are twist knots, which is attributed to the dissertation of Fung by Ozawa; see~\cite[Theorem~2.8]{Ozawa:ascending}, indicated with~\cite{Fung:immersions}.
\end{proof}

The last column of Tables~\ref{tab:8} and~\ref{tab:9} contains the value of an invariant that we call \emph{roller-coaster crossing number}, which is the smallest natural number $c$ such that $K$ admits a diagram with $c$ crossings for which the roller-coaster algorithm realizes the ascending number of the knot.
The notation $\{c, 12+\}$ means that this invariant takes value $c$ if the ascending number matches the upper bound given in the table and is at least 12 otherwise.

Our computer program is able to analyze alternating knots with larger crossing numbers.
By comparing the ascending numbers calculated by our program with the data on KnotInfo~\cite{KnotInfo}, we can prove Theorem~\ref{thm:alt_12}.

\begin{proof}[Proof of Theorem~\ref{thm:alt_12}]
	For part (1), see Table~\ref{tab:8}.
	The knots $\mathbf{12a_{181}}$ and $\mathbf{12a_{477}}$ each have unknotting number 2 or 3, according to KnotInfo~\cite{KnotInfo}.
	The knot $\mathbf{12a_{181}}$ has a diagram with DT code $[8, 6, 16, 10, 24, 14, 20, 18, 4, 22, 12, 2]$, which can be seen to have ascending number 3; the knot $\mathbf{12a_{477}}$ has a diagram with DT code $[14, 24, 22, 20, 6, 4, 2, 12, 10, 8, 18, 16]$, which can be seen to have ascending number 3.
\end{proof}

\renewcommand{\arraystretch}{1.6}

\begin{table}[!ht]
    \centering
    \resizebox*{!}{1.1\textwidth}{%
	\begin{tabular}{|c|c|c|c|c|c|c|c|}
    \hline
        \textbf{Name} & \textbf{Alternating} & \textbf{Unknotting} & \textbf{Ascending} & \textbf{Lower} & \textbf{Property:}  & \textbf{DT code} & \textbf{Roller-coaster} \\
          & & \textbf{Number} & \textbf{Number} & \textbf{Bound} & \textbf{SRC or RC} & & \textbf{Crossing} \\
          & & & & & & & \textbf{Number} \\ \hline
        $3_1$ & Y & 1 & 1 & $u(K)$ & SRC & [4, 6, 2] & 3 \\ \hline
        $4_1$ & Y & 1 & 1 & $u(K)$ & SRC & [4, 6, 8, 2] & 4 \\ \hline
        $5_1$ & Y & 2 & 2 & $u(K)$ & SRC & [6, 8, 10, 2, 4] & 5 \\ \hline
        $5_2$ & Y & 1 & 1 & $u(K)$ & RC & [-8, 10, 2, -12, 6, 4] & 6 \\ \hline
        $6_1$ & Y & 1 & 1 & $u(K)$ & RC & [-8, 12, 2, -10, -14, 4, 6] & 7 \\ \hline
        $6_2$ & Y & 1 & 2 & \cite{Fung:immersions} & neither & [4, 8, 10, 12, 2, 6] & 6 \\ \hline
        $6_3$ & Y & 1 & 2 & \cite{Fung:immersions} & neither & [4, 8, 10, 2, 12, 6] & 6 \\ \hline
        $7_1$ & Y & 3 & 3 & $u(K)$ & SRC & [8, 10, 12, 14, 2, 4, 6] & 7 \\ \hline
        $7_2$ & Y & 1 & 1 & $u(K)$ & RC & [-8, -14, -12, -18, 6, 2, -16, -4, 10] & 9 \\ \hline
        $7_3$ & Y & 2 & 2 & $u(K)$ & RC & [-8, 12, 14, 10, -2, 16, 4, 6] & 8 \\ \hline
        $7_4$ & Y & 2 & 2 & $u(K)$ & RC & [-14, 12, 10, -16, 2, 4, 6, -8] & 8 \\ \hline
        $7_5$ & Y & 2 & 2 & $u(K)$ & RC & [-8, 12, 14, -10, -2, -16, 4, 6] & 8 \\ \hline
        $7_6$ & Y & 1 & 2 & \cite{Fung:immersions} & neither & [8, 14, 12, 2, 6, 4, 10] & 7 \\ \hline
        $7_7$ & Y & 1 & 2 & \cite{Fung:immersions} & neither & [8, 12, 14, 2, 6, 4, 10] & 7 \\ \hline
        $8_1$ & Y & 1 & 1 & $u(K)$ & RC & [12, 6, 16, 18, -2, 14, 20, -8, 4, -10] & 10 \\ \hline
        $8_2$ & Y & 2 & 3 & \cite{Higa:conway} & neither & [10, 12, 14, 4, 16, 2, 6, 8] & 9 \\ \hline
        $8_3$ & Y & 2 & 2 & $u(K)$ & RC & [-8, 14, 12, 16, -18, 6, 4, 10, -2] & 9 \\ \hline
        $8_4$ & Y & 2 & 2 & $u(K)$ & RC & [-12, 8, 16, 14, 18, -2, -10, 4, 6] & 9 \\ \hline
        $8_5$ & Y & 2 & 3 & \cite{Higa:conway} & neither & [10, 14, 16, 12, 2, 6, 8, 4] & 9 \\ \hline
        $8_6$ & Y & 2 & 2 & $u(K)$ & RC & [8, 16, -12, 10, 18, 14, -6, 2, 4] & 9 \\ \hline
        $8_7$ & Y & 1 & [2, 3] & $u(K)$ & neither & [-10, -18, 14, 16, -12, -2, 4, 6, 8] & \{9, 12+\} \\ \hline
        $8_8$ & Y & 2 & 2 & $u(K)$ & RC & [14, -12, 18, 16, 6, -2, -4, 10, 8] & 9 \\ \hline
        $8_9$ & Y & 1 & [2, 3] & $u(K)$ & neither & [10, 12, 16, 14, 4, 2, 6, 8] & \{9, 12+\} \\ \hline
        $8_{10}$ & Y & 2 & [2, 3] & $u(K)$ & ? & [6, 16, 12, 2, 14, 4, 8, 10] & \{9, 12+\} \\ \hline
        $8_{11}$ & Y & 1 & 2 & \cite{Fung:immersions} & neither & [10, 12, -16, 18, 2, 8, 6, -4, 14] & 9 \\ \hline
        $8_{12}$ & Y & 2 & 2 & $u(K)$ & SRC & [10, 16, 14, 4, 2, 8, 6, 12] & 8 \\ \hline
        $8_{13}$ & Y & 1 & 2 & \cite{Fung:immersions} & neither & [-8, 12, 10, -18, 16, 14, 4, 6, 2] & 9 \\ \hline
        $8_{14}$ & Y & 1 & 2 & \cite{Fung:immersions} & neither & [-4, -10, -12, 16, -18, -2, 8, 6, 14] & 9 \\ \hline
        $8_{15}$ & Y & 2 & 2 & $u(K)$ & RC & [-10, 16, 12, 4, -18, 8, 6, 2, 14] & 9 \\ \hline
        $8_{16}$ & Y & 2 & 2 & $u(K)$ & RC & [14, -16, 20, 18, -2, 4, 6, 10, 8, 12] & 10 \\ \hline
        $8_{17}$ & Y & 1 & 2 & \cite{Fung:immersions} & neither & [-8, 10, 12, -14, 16, 18, 20, 2, 4, 6] & 10 \\ \hline
        $8_{18}$ & Y & 2 & 2 & $u(K)$ & SRC & [12, 14, 16, 2, 4, 6, 8, 10] & 8 \\ \hline
        $8_{19}$ & N & 3 & 3 & $u(K)$ & SRC & [12, -14, 16, -2, 4, -6, 8, -10] & 8 \\ \hline
        $8_{20}$ & N & 1 & 2 & \cite{Fung:immersions} & neither & [-12, 14, -16, 2, 4, -6, 8, 10] & 8 \\ \hline
        $8_{21}$ & N & 1 & 2 & \cite{Fung:immersions} & neither & [12, 6, 10, -16, 4, 14, 2, -8] & 8 \\ \hline
    \end{tabular}
    }
    \caption{Determining which knots with at most 8 crossings have the roller-coaster property or the strong roller-coaster property}
    \label{tab:8}
\end{table}

\begin{table}[!ht]
    \centering
    \resizebox*{!}{\dimexpr\textheight-5\baselineskip\relax}{%
	\begin{tabular}{|c|c|c|c|c|c|c|c|}
    \hline
        \textbf{Name} & \textbf{Alternating} & \textbf{Unknotting} & \textbf{Ascending} & \textbf{Lower} & \textbf{Property:}  & \textbf{DT code} & \textbf{Roller-coaster} \\
          & & \textbf{Number} & \textbf{Number} & \textbf{Bound} & \textbf{SRC or RC} & & \textbf{Crossing} \\
          & & & & & & & \textbf{Number} \\ \hline
        $9_1$ & Y & 4 & 4 & $u(K)$ & SRC & [10, 12, 14, 16, 18, 2, 4, 6, 8] & 9 \\ \hline
        $9_2$ & Y & 1 & 1 & $u(K)$ & RC & [16, 10, -20, -4, 22, 18, -2, 24, -14, -8, -6, -12] & 12 \\ \hline
        $9_3$ & Y & 3 & 3 & $u(K)$ & RC & [-10, 14, 16, 18, 12, -2, 20, 4, 6, 8] & 10 \\ \hline
        $9_4$ & Y & 2 & 2 & $u(K)$ & RC & [14, -10, 18, 20, -16, -22, 2, 12, 4, 6, 8] & 11 \\ \hline
        $9_5$ & Y & 2 & 2 & $u(K)$ & RC & [-20, -18, 14, 16, -2, 22, 4, 8, 6, -12, 10] & 11 \\ \hline
        $9_6$ & Y & 3 & 3 & $u(K)$ & RC & [-10, 14, 16, 18, -12, -2, -20, 4, 6, 8] & 10 \\ \hline
        $9_7$ & Y & 2 & 2 & $u(K)$ & RC & [18, 16, 14, 4, 22, 20, -2, 6, -8, 10, 12] & 11 \\ \hline
        $9_8$ & Y & 2 & 2 & $u(K)$ & RC & [4, 14, 12, -18, -6, 16, 2, 20, -8, -10] & 10 \\ \hline
        $9_9$ & Y & 3 & 3 & $u(K)$ & RC & [10, 12, 20, -16, -18, 2, 4, -6, -8, 14] & 10 \\ \hline
        $9_{10}$ & Y & 3 & 3 & $u(K)$ & RC & [8, -14, 12, 18, 20, 2, 16, -4, 6, 10] & 10 \\ \hline
        $9_{11}$ & Y & 2 & 3 & \cite{Higa:conway} & neither & [12, 18, 14, 16, 4, 2, 10, 6, 8] & \{9, 12+\} \\ \hline
        $9_{12}$ & Y & 1 & 2 & \cite{Fung:immersions} & neither & [14, -22, 20, 18, 6, 4, 2, 12, 10, 8, 16] & 11 \\ \hline
        $9_{13}$ & Y & 3 & 3 & $u(K)$ & RC & [-10, -12, 14, 16, -2, -20, 18, 6, 4, -8] & 10 \\ \hline
        $9_{14}$ & Y & 1 & 2 & \cite{Fung:immersions} & neither & [-10, 8, 14, 16, 12, -2, 20, 18, 6, 4] & 10 \\ \hline
        $9_{15}$ & Y & 2 & 2 & $u(K)$ & RC & [12, 20, -16, -18, 4, 2, 10, 8, -6, 14] & 10 \\ \hline
        $9_{16}$ & Y & 3 & 3 & $u(K)$ & RC & [14, 10, -16, -18, 2, 4, 20, 8, -6, 12] & 10 \\ \hline
        $9_{17}$ & Y & 2 & [2, 3] & $u(K)$ & ? & [12, 14, 16, 4, 18, 2, 10, 6, 8] & \{9, 12+\} \\ \hline
        $9_{18}$ & Y & 2 & 2 & $u(K)$ & RC & [12, -14, 18, -20, -22, 2, 10, 6, 4, 16, -8] & 11 \\ \hline
        $9_{19}$ & Y & 1 & 2 & \cite{Fung:immersions} & neither & [6, 8, 12, 20, -16, 18, 4, -10, -14, 2] & 10 \\ \hline
        $9_{20}$ & Y & 2 & 3 & \cite{Higa:conway} & neither & [12, 14, 18, 16, 6, 4, 2, 10, 8] & 9 \\ \hline
        $9_{21}$ & Y & 1 & 2 & \cite{Fung:immersions} & neither & [12, 20, 14, -16, 4, 2, 10, -18, -6, 8] & 10 \\ \hline
        $9_{22}$ & Y & 1 & [2, 3] & \cite{Fung:immersions} & neither & [6, 8, 12, 18, 14, 16, 4, 10, 2] & \{9, 12+\} \\ \hline
        $9_{23}$ & Y & 2 & 2 & $u(K)$ & RC & [-12, 14, -18, 20, 22, -6, 16, 4, 2, -10, 8] & 11 \\ \hline
        $9_{24}$ & Y & 1 & [2, 3] & \cite{Fung:immersions} & neither & [8, 16, 12, 10, 18, 14, 6, 2, 4] & \{9, 12+\} \\ \hline
        $9_{25}$ & Y & 2 & 2 & $u(K)$ & RC & [14, 20, -10, 18, -4, -6, 2, 12, 8, 16] & 10 \\ \hline
        $9_{26}$ & Y & 1 & [2, 3] & \cite{Fung:immersions} & neither & [8, 10, 14, 16, 12, 2, 18, 6, 4] & \{9, 12+\} \\ \hline
        $9_{27}$ & Y & 1 & [2, 3] & \cite{Fung:immersions} & neither & [10, 12, 16, 18, 2, 8, 6, 4, 14] & \{9, 12+\} \\ \hline
        $9_{28}$ & Y & 1 & [2, 3] & \cite{Fung:immersions} & neither & [10, 16, 12, 4, 18, 8, 6, 2, 14] & \{9, 12+\} \\ \hline
        $9_{29}$ & Y & 2 & [2, 3] & $u(K)$ & ? & [8, 18, -12, 2, 16, -4, 10, 20, 6, 14] & \{10, 12+\} \\ \hline
        $9_{30}$ & Y & 1 & [2, 3] & \cite{Fung:immersions} & neither & [8, 14, 18, 2, 6, 16, 10, 4, 12] & \{9, 12+\} \\ \hline
        $9_{31}$ & Y & 2 & [2, 3] & $u(K)$ & ? & [10, 6, 14, 12, 18, 16, 4, 8, 2] & \{9, 12+\} \\ \hline
        $9_{32}$ & Y & 2 & 2 & $u(K)$ & RC & [16, 18, -20, 22, 2, 6, -4, 10, 8, 14, 12] & 11 \\ \hline
        $9_{33}$ & Y & 1 & 2 & \cite{Fung:immersions} & neither & [-10, -16, 14, 12, -2, 18, 20, 22, 8, 6, 4] & 11 \\ \hline
        $9_{34}$ & Y & 1 & 2 & \cite{Fung:immersions} & neither & [16, -10, 20, 18, -2, 8, 6, 4, 12, 14] & 10 \\ \hline
        $9_{35}$ & Y & 3 & 3 & $u(K)$ & RC & [-12, 14, 16, 20, 18, -2, -10, 6, 4, 8] & 10 \\ \hline
        $9_{36}$ & Y & 2 & 3 & \cite{Higa:conway} & neither & [14, 18, 10, 16, 4, 6, 2, 12, 8] & 9 \\ \hline
        $9_{37}$ & Y & 2 & 2 & $u(K)$ & RC & [-18, 10, 14, -20, 12, 2, 16, 4, 6, -8] & 10 \\ \hline
        $9_{38}$ & Y & 3 & 3 & $u(K)$ & RC & [8, 12, -18, 2, 14, 20, 6, 10, -4, 16] & 10 \\ \hline
        $9_{39}$ & Y & 1 & [2, 3] & \cite{Fung:immersions} & neither & [-10, 16, 22, 20, 14, -4, 6, 18, 2, 8, 12] & \{11, 12+\} \\ \hline
        $9_{40}$ & Y & 2 & [2, 3] & $u(K)$ & ? & [-16, 20, 18, -22, 14, 4, 2, 8, 6, 12, 10] & \{11, 12+\} \\ \hline
        $9_{41}$ & Y & 2 & [2, 3] & $u(K)$ & ? & [6, 10, 14, 12, 16, 2, 18, 4, 8] & \{9, 12+\} \\ \hline
        $9_{42}$ & N & 1 & 2 & \cite{Fung:immersions} & neither & [4, 8, -14, 18, 2, -16, -6, -10, -12] & 9 \\ \hline
        $9_{43}$ & N & 2 & 3 & \cite{Higa:conway} & neither & [-14, -18, 10, 16, 4, 6, -2, -12, 8] & 9 \\ \hline
        $9_{44}$ & N & 1 & 2 & \cite{Fung:immersions} & neither & [8, 18, -12, 2, 6, 16, -4, 10, 14] & 9 \\ \hline
        $9_{45}$ & N & 1 & 2 & \cite{Fung:immersions} & neither & [8, 20, -14, 2, 6, 18, -16, -4, 10, 12] & 10 \\ \hline
        $9_{46}$ & N & 2 & 2 & $u(K)$ & SRC & [-12, 6, 14, 18, 16, -2, -10, 4, 8] & 9 \\ \hline
        $9_{47}$ & N & 2 & 2 & $u(K)$ & SRC & [-14, 8, 10, 12, 18, 16, -2, 6, 4] & 9 \\ \hline
        $9_{48}$ & N & 2 & 2 & $u(K)$ & SRC & [-10, -8, 12, 16, -2, 14, 4, 18, 6] & 9 \\ \hline
        $9_{49}$ & N & 3 & 3 & $u(K)$ & SRC & [14, -8, 12, 16, -2, 18, 4, 10, 6] & 9 \\ \hline
    \end{tabular}
    }
    \caption{Determining which knots with 9 crossings have the roller-coaster property or the strong roller-coaster property}
    \label{tab:9}
\end{table}

\newpage

\

\newpage

\bibliographystyle{amsalpha}
\bibliography{roller-coaster.bib}

\end{document}